\documentclass[11pt,letterpaper]{amsart}
\usepackage{fancyhdr}
\usepackage{bm}
\usepackage{hyperref}
\usepackage{graphicx} 
\usepackage{nicematrix}
\usepackage{amsthm,amssymb,amsmath}
\usepackage[T1]{fontenc}
\usepackage[utf8]{inputenc}
\usepackage{hyperref}
\usepackage{lipsum}
\usepackage{mathrsfs}
\usepackage{enumitem}
\usepackage{stmaryrd}
\usepackage{setspace}
\usepackage{yfonts}
\usepackage{float}
\usepackage{mathtools}
\usepackage{parskip}
\usepackage[all,cmtip]{xy}
\usepackage{tikz-cd}
\tikzcdset{row sep/normal=50pt, column sep/normal=50pt}

\newtheorem{lemma}{Lemma}[section]
\newtheorem{remark}[lemma]{Remark}
\newtheorem{theorem}[lemma]{Theorem}
\newtheorem{theorem*}{Theorem}

\newtheorem{example*}[lemma]{Example}
\newtheorem{proposition}[lemma]{Proposition}
\newtheorem{corollary}[lemma]{Corollary}
\newtheorem{claim}[lemma]{Claim}

\newtheorem{manualtheoreminner}{Theorem}
\newenvironment{manualtheorem}[1]{%
  \renewcommand\themanualtheoreminner{#1}%
  \manualtheoreminner
}{\endmanualtheoreminner}
\newtheorem{manualsupplementinner}{Supplement}
\newenvironment{manualsupplement}[1]{%
  \renewcommand\themanualsupplementinner{#1}%
  \manualsupplementinner
}{\endmanualsupplementinner}

\usepackage{blindtext}
\usepackage[margin=1.1in]{geometry}
\usepackage{blindtext}

\usepackage[backend=biber, style=numeric, sorting=nyt]{biblatex}
\title{Tangent bundles of Quot schemes of vector bundles on curves}
\author{\small{Ashima Bansal, Supravat Sarkar, Shivam Vats}}
\date{}
\begin{document}

\begin{abstract}
   For a vector bundle $E$ on a smooth projective curve $C$, one defines the Quot scheme $Q:=Q_d(E,C)$ parametrizing subsheaves of $E$ having length $d$ torsion quotients. We describe the indecomposable components of the tangent bundle $T_Q$ of $Q$, thus characterizing when $T_Q$ is indecomposable. We also characterize when $T_Q$ is simple; in fact, we completely describe the endomorphism rings of the indecomposable components of $T_Q$. As an application, we determine when two products of such Quot schemes can be isomorphic.
\end{abstract}
\maketitle
\begin{center}
\textbf{Keywords}: Quot scheme, tangent bundle
\end{center}
\begin{center}
\textbf{MSC Number: 14C05, 14F06} 
\end{center}

\section{Introduction}
We work throughout over the field $\mathbb{C}$ of complex numbers. Given a nonzero vector bundle $E$ of rank $r$ on a smooth projective curve $C$, and a positive integer $d$, $Q_d(E,C)$ is the Quot scheme parametrizing coherent subsheaves $K$ of $E$ such that $E/K$ is torsion of length $d$. This is a smooth projective variety of dimension $rd$. When $E$ is a line bundle, the Quot scheme is the symmetric power $C^{(d)}$, the quotient of $C^d$ under the action of the symmetric group $S_d$ by permuting the coordinates. When $d=1$, this Quot scheme is the projective bundle $\mathbb{P}_C(E)$.

In recent years there has been a lot of research about $Q:=Q_d(E,C)$. \cite{Infdefor1} and \cite{Infdefor2} study its deformation theory, \cite{GangopadhyaySebastianFundamental} studies its fundamental group scheme, \cite{gangopadhyay2019automorphisms},  \cite{Torelli} study its automorphisms and Torelli-type theorems and \cite{gangopadhyay2026birational} studies its birational geometry. The study of vector bundles, or more generally sheaves on $Q$ is also an active area of research.  \cite{GangopadhyaySebastianPicard} studies Picard groups and \cite{GangopadhyaySebastian} studies nef cones of $Q$, \cite{KrugExtension}, \cite{marian2026cohomology}, \cite{oprea2023euler} and \cite{gangopadhyay2018stability} study tautological bundles on $Q$ and \cite{marian2024derived} studies the derived category of $Q$.

A very natural and important vector bundle associated to any smooth projective variety is its tangent bundle. The study of the tangent bundle often reveals important geometric properties of the variety; Mori's work in \cite{mori1979projective} is a striking example. Tangent bundles of varieties have been studied in the literature from several viewpoints; see \cite{mehta1987varieties} or \cite{tian1992stability}, for example. The goal of this article is to study the tangent bundle of the Quot scheme $Q$. A nonzero vector bundle is called \textit{decomposable} if it is a direct sum of two nonzero vector bundles, and \textit{indecomposable} otherwise. By \cite[Theorem 3]{atiyah1956krull}, every vector bundle on a projective variety can be written as a direct sum of indecomposable vector bundles, which are uniquely determined up to isomorphism. As can be motivated from a representation-theoretic viewpoint, it is interesting to study these indecomposable components for interesting vector bundles. In \cite{SarkarTangentHilbert}, the indecomposable components of the tangent bundle of the Hilbert scheme of points on a smooth projective surface were completely described. In this paper, we solve the analogous problem for Quot schemes.

A nonzero vector bundle $E$ is called \textit{simple} if $End(E)=\mathbb{C}\cdot id$. Simple vector bundles are indecomposable, but the converse is not true. For each indecomposable component of $T_Q$, we also determine when it is simple, and completely describe its endomorphism algebra when it is not simple.
\begin{manualtheorem}{A}\label{A}
Let $E$ be a nonzero vector bundle on a smooth projective curve $C$ of genus $g$. Also, let $d\geq 2$ be an integer and $Q=Q_d(E,C)$. Then the following hold:
\begin{enumerate}
    \item[(a)] If $g=0$, then $T_Q$ is simple.
    
    \item[(b)] If $g=1$ and $E$ is not semihomogeneous, then $T_Q$ is indecomposable but not simple.
    
    \item[(c)] If $g=1$ and $E$ is semihomogeneous, then the Albanese map $    a:Q\longrightarrow \operatorname{Alb}(Q)\cong C$
    is smooth, $T_a$ is simple, and $    T_Q\cong \mathcal{O}_Q\oplus T_a.$ 
    
    \item[(d)] If $g\geq 2$ and $E$ is not simple, then $T_Q$ is indecomposable but not simple.
    
    \item[(e)] If $g\geq 2$ and $E$ is simple, then $T_Q$ is simple.
\end{enumerate}
Moreover, in (b) and (d), we have a $\mathbb{C}$-algebra isomorphism
\[
\operatorname{End}(T_Q)
\cong
\mathbb{C}\cdot\operatorname{id}
\oplus
\left(H^0(C,\omega_C)\otimes H^0(C,\operatorname{ad}E)\right),
\]
where the second summand is a square-zero ideal.
\end{manualtheorem}
Here $\omega_C$ is the canonical bundle of $C$, and $\operatorname{ad}E=\operatorname{\mathcal{E}nd}(E)/\mathcal{O}_C,$ where $\mathcal{O}_C\subset \operatorname{\mathcal{E}nd}(E)$ is induced by the identity endomorphism. Note that we include the case when $E$ is a line bundle, thus we are also describe the indecomposable components of the tangent bundle of the symmetric power $C^{(d)}$. This special case is stated in a more concise form in Proposition \ref{r=1}.

For $d=1$, the Quot scheme is the projective bundle $\mathbb{P}_C(E)$, and behaviour of tangent bundle is different in this case.
\begin{manualsupplement}{A}\label{PE}
Let $E$ be a vector bundle of rank $\geq 2$ on a smooth projective curve $C$. Let $Q=\mathbb P_C(E)$ and $Q\xrightarrow{\varphi}C$ be the projection. Then:

\begin{enumerate}
    \item[(a)] $T_Q$ is decomposable if and only if all indecomposable components of $E$ have the same slope. In this case, $T_Q\cong T_{\varphi}\oplus\varphi^*T_C$, and both summands are simple.

    \item[(b)] If $T_Q$ is indecomposable, then
    \[
    \operatorname{End}(T_Q)
    \cong
    \mathbb C\cdot\operatorname{id}
    \oplus H^0(C,\omega_C\otimes\operatorname{ad}(E)),
    \]
    where the second summand is a square-zero ideal.

    \item[(c)] $T_Q$ is simple if and only if $C\cong\mathbb P^1$, and $E$ is balanced but not trivial up to line bundle twist.
\end{enumerate}
    Here a vector bundle $E\cong \oplus_i\mathcal{O}_{\mathbb{P}^1}(a_i)$ on $\mathbb{P}^1$ is called \textit{balanced} if $|a_i-a_j|\leq 1$ for all $i$ and $j$.
\end{manualsupplement}  
When $C$ is an elliptic curve with origin $O$, one can define the Kummer-Quot scheme as in \cite[Section 5]{Torelli} to be the fibre over $O$ of the Albanese map $Q_d(E,C)\to C$. We also prove an analogous result for the tangent bundle of the Kummer-Quot scheme.

\begin{manualtheorem}{A$'$}\label{T_Q'}
Let $E$ be a nonzero vector bundle on an elliptic curve $C$
with origin $O$. Also, let $d\geq 2$ be an integer and $Q'=Q_d'(E,C).$
Then the following hold:

\begin{enumerate}
    \item[(a)] If $d\geq 3$, then $T_{Q'}$ is simple.

    \item[(b)] If $d=2$, then $T_{Q'}$ is indecomposable, and
    \[
    \operatorname{End}(T_{Q'})
    \cong
    \mathbb C\cdot\operatorname{id}
    \oplus
    \operatorname{Hom}_C(E(4\cdot O),E),
    \]
    where the second summand is a square-zero ideal.
\end{enumerate}
\end{manualtheorem}
Here we assume $d\geq 2$ as for $d=1$, $Q'$ is a projective space, so $T_{Q'}$ is simple when it is nonzero.

As an application of our results, we describe when two products of Quot schemes can be isomorphic.
\begin{manualtheorem}{B}\label{B}
Let $k$ and $l$ be positive integers. For each $1\leq i\leq k$, let $d_i\geq 2$ be an integer, and $E_i$ a vector bundle of rank $\geq 2$ on a smooth projective curve $C_i$. Similarly, for each $1\leq j\leq l$, let $d_j'\geq 2$ be an integer, and $E_j'$ a vector bundle of rank $\geq 2$ on a smooth projective curve $C_j'$. Suppose
\[
\prod_{i=1}^k Q_{d_i}(E_i,C_i)\cong \prod_{j=1}^l Q_{d_j'}(E_j',C_j').
\]
Then $k=l$, and up to renumbering, we have for each $i$, $d_i=d_i'$, and $E_i\cong h_i^*E_i'\otimes L_i$ for an isomorphism $h_i:C_i\to C_i'$ and a line bundle $L_i$ on $C_i$.
\end{manualtheorem}
In \cite[Theorem B and C]{SarkarTangentHilbert}, similar results were obtained for symmetric powers and Hilbert schemes of points on a smooth projective surface. 

At the end of \S 5, we discuss a few general results from which one can completely describe automorphisms of certain products of Quot schemes, as another application of Theorem \ref{A}.
\section{Notation and conventions}
\begin{itemize}
\item For a commutative ring $R$ and a positive integer $d$, the ring structure on $R^{\oplus d}$ will be assumed to be the product ring structure, unless otherwise stated.
\item For a group $G$ acting on a ring $R$, the fixed ring $R^G$ consists of all $x\in R$ fixed by $G$.
\item For a projective variety $X$ we denote the identity component of the automorphism group scheme of $X$ by $\operatorname{Aut}^{0}(X)$.
\item For a group homomorphism $N:G\to H$, we denote the kernel and image of $N$ by $\ker N$ and $\operatorname{im}N$, respectively.
\item Let $C$ be an elliptic curve. A vector bundle $E$ on $C$ is said to be \textit{semihomogeneous} if for all $x\in C$, there is a line bundle $L_x$ on $C$ such that $t_x^*E\cong E\otimes L_x,$
where $t_x$ is the translation by $x$. This is equivalent to the assertion that the natural map $\operatorname{Aut}^0(\mathbb P_C(E))
\longrightarrow
\operatorname{Aut}^0(C)$ is surjective. In fact, $E$ is semihomogeneous if and only if $E$ is semistable; see
\cite[Theorem~2 and the discussion on page~2]{mehta1984semistable}.
\item A \textit{torsion bundle over a torus} is a variety which is the total space of a fibre bundle with finite structure group over a positive dimensional abelian variety. See \cite{horst1985decomposition} for more details. If $W$ is a projective variety that is a torsion bundle over a torus, then $W$ can be written as $(A\times X)/G$, where $G$ is a finite subgroup of a positive dimensional abelian variety $A$, with an action of $G$ on $X$, so that $G$ acts on $A\times X$ diagonally (see \cite[page~2]{ballico2003real}). Since $A/G$ is an abelian variety, \cite[Lemma~6.2]{SarkarTangentHilbert} shows that $\mathcal O_W$ is a direct summand of $T_W$. Also, the action of $A$ on $A\times X$ via translation in the $A$-factor descends to an action of $A$ on $W$. Thus, $\operatorname{Aut}^{0}(W)$ contains a positive dimensional abelian variety.
\item A positive dimensional projective variety is called \textit{decomposable} if it is isomorphic to the product of two positive dimensional projective varieties, and \textit{indecomposable} otherwise.
\item For a smooth variety $X$, its tangent bundle will be denoted by $T_X$. For a morphism $f:X\to Y$ of smooth varieties, the kernel of $df:T_X\to f^*T_Y$, will be denoted by $T_f.$ This is the so-called relative tangent bundle when $f$ is smooth.
\item A proper surjective morphism between normal varieties $f:X\to Y$ is called a \textit{contraction} if $f_*\mathcal{O}_X=\mathcal{O}_Y$. This is equivalent to the assertion that $f$ has connected fibres.
\item For a smooth projective curve $C$ and a positive integer $d$, $\Sigma_d(C) \hookrightarrow C \times C^{(d)}$
denotes the universal divisor, consisting of pairs $(x,D)$ with $x\leq D$. Let $\Sigma_d(C) \xrightarrow{q} C$
and 
$\Sigma_d(C) \xrightarrow{p} C^{(d)}$ be the projections. Given a vector bundle $E$ on $C$, we define the secant bundle $E^{[d]}:=p_*q^*E,$
which is a vector bundle on $C^{(d)}$, as $p$ is finite flat.
\item For an elliptic curve $C$ with origin $O$ and a positive integer $d$, the sum map $C^{(d)}\xlongrightarrow{b} C$ is a $\mathbb{P}^{d-1}$-bundle. We let $C^{((d))}\cong \mathbb{P}^{d-1}$ to be the fibre of $b$ over $O$. Let $\Sigma'_d(C) \subset C\times C^{((d))}$
be the intersection of $\Sigma_d(C)$ with $C\times C^{((d))}$, and $\Sigma'_d(C) \xrightarrow{q'} C,$
$\Sigma'_d(C) \xrightarrow{p'} C^{((d))}$
be the projections. For a vector bundle $E$ on $C$, define a vector bundle $E^{[[d]]}$ on
$C^{((d))}$ by $E^{[[d]]}=p'_*(q')^*E.$ See \cite[Section 5]{Torelli} for more details.
\item For a nonzero vector bundle $E$ on a variety $X$, let $\operatorname{ad}(E)
:=\operatorname{\mathcal{E}nd}(E)/\mathcal{O}_X,$ where $\mathcal{O}_X\subset \operatorname{\mathcal{E}nd}(E)$ is induced by the identity endomorphism.
 This $\operatorname{ad}(E)$ is isomorphic to the vector bundle $\mathfrak{sl}(E)$ of trace zero endomorphisms, and we have 
\begin{equation}\label{ad E}
    \operatorname{\mathcal{E}nd}(E)\cong \mathcal{O}_X\oplus \operatorname{ad}(E),
\end{equation} see \cite[Page 4]{Sarkar2026Automorphisms}.
\end{itemize}
\section{Some preliminary results}
In this section, we prove a few results, that will be needed in the proof of the main results of the paper.

\begin{lemma}\label{simple}
We have the following:
\begin{enumerate}
    \item $T_{\mathbb P^n}$ is simple for all $n\geq 1$.

    \item Let $    f:X\longrightarrow Y$
    be a smooth contraction of smooth projective varieties. Suppose
    $T_{f^{-1}(y)}$ is simple for every $y\in Y$. Then $T_f$ is simple.
\end{enumerate}
\end{lemma}

\begin{proof}
(1) is \cite[Lemma~4.1.2, Chapter~1, Section~4.1]{okonek1980vector}.
Let us prove (2). By Grauert's theorem, $f_*\operatorname{\mathcal End}(T_f)$
is a line bundle on $Y$, with $\operatorname{End}(T_{f^{-1}(y)})$
being the fibre over $y\in Y$. The natural map
\[
\mathcal O_Y\longrightarrow f_*\operatorname{\mathcal End}(T_f)
\]
given by the identity endomorphism is an isomorphism on each fibre, hence is
an isomorphism. Now taking $H^0$, we get the result.
\end{proof}
The following is the special case of Theorem \ref{A} where $\operatorname{rank }E=1.$
\begin{proposition}\label{r=1}
Let $C$ be a smooth projective curve of genus $g$, and $d\geq 2$ an integer. Then the following hold:
\begin{enumerate}
    \item If $g=1$, then the sum map $b:C^{(d)}\longrightarrow C$
is smooth, $T_b$ is simple, and $T_{C^{(d)}}\cong \mathcal{O}_{C^{(d)}}\oplus T_b$.
    
    \item In all other cases, $T_{C^{(d)}}$ is simple.
\end{enumerate}
\end{proposition}
\begin{proof}
\underline{$(1):$} By \cite[Lemma 5.1]{Torelli} applied to $E=\mathcal{O}_C$, we only need to show $T_b$ is simple. This follows as the restriction of $T_b$ to a fibre of $b$ is $T_{\mathbb{P}^{d-1}}$, which is simple.

\underline{$(2):$} If $g(C)=0$, then we are done as $C^{(d)}\cong\mathbb{P}^d$ and $T_{\mathbb{P}^d}$ is simple.
Now assume $g(C)\geq 2$. By \cite[Exercise 3.22]{harris1998moduli}, $\Omega_{C^{(d)}}\cong \omega_C^{[d]}.$
Now $\omega_C$ is a nontrivial line bundle, in particular simple. So by \cite[Theorem 1.1]{KrugExtension}, $\omega_C^{[d]}$
is simple. Thus, $\Omega_{C^{(d)}}$ is simple, and hence so is its dual $T_{C^{(d)}}$.
\end{proof}
\begin{lemma}\label{Sx}
Let $C$ be a smooth projective curve, $d\geq 2$ an integer, and $S=C^{(d)}$. For $x\in C$, let
\[
S_x=\{[D]\in S\mid x\leq D\}.
\]
Then the natural map $H^0(S,\Omega_S)\longrightarrow H^0(S_x,\Omega_S|_{S_x})$
is an isomorphism.
\end{lemma}
\begin{proof}
We have the conormal exact sequence
\begin{equation}\label{conormal}
0\longrightarrow \mathcal{O}_{S}(-S_x)|_{S_x}
\longrightarrow \Omega_S|_{S_x}
\longrightarrow \Omega_{S_x}\longrightarrow 0.
\end{equation}
By \cite[Chapter VII, Proposition 2.2]{ACGH}, $S_x$ is an ample divisor in $S$. So $\mathcal{O}_{S}(-S_x)|_{S_x}$ is anti-ample; in particular, $H^0(S_x,\mathcal{O}_{S}(-S_x)|_{S_x})=0.$
By \eqref{conormal}, we have an injection
\[
H^0(S_x,\Omega_S|_{S_x})\hookrightarrow H^0(S_x,\Omega_{S_x}).
\]

The map $C^{(d-1)}\xhookrightarrow{+x}C^{(d)}$
given by $D\mapsto D+x$ has image $S_x$, and we have a commutative diagram
\[
\begin{tikzcd}
C^{(d-1)}
  \arrow[r,hook, "{+x}"]
  \arrow[d]
&
C^{(d)}
  \arrow[d]
\\
\operatorname{Pic}^{d-1}(C)
  \arrow[r, "\cong" above, "{-\otimes\mathcal{O}_C(x)}" below]
&
\operatorname{Pic}^{d}(C),
\end{tikzcd}
\]
where the vertical maps are the Abel--Jacobi maps, which are the same as the Albanese maps of $C^{(d-1)}$ and $C^{(d)}$. So, all the maps in the diagram, except possibly the $+x$ map, induce isomorphisms on the spaces of $1$-forms under pullback. Hence the $+x$ map also induces an isomorphism on the spaces of $1$-forms. Thus we have an isomorphism $H^0(S,\Omega_S)\to H^0(S_x,\Omega_{S_x})$. Thus we have a commutative diagram
\[
\begin{tikzcd}
H^0(S,\Omega_S)
    \arrow[r]
    \arrow[dr, "\cong"']
&
H^0(S_x,\Omega_S|_{S_x})
    \arrow[hook, d]
\\
&
H^0(S_x,\Omega_{S_x})
\end{tikzcd}
\]
As the diagonal map in this diagram is an isomorphism and the vertical map is injective, one easily sees that the horizontal map must be an isomorphism, thereby proving the Lemma.
\end{proof}
\begin{lemma}\label{homtf}
Let $F$ be a vector bundle on a smooth projective curve $C$. Let $d\geq 2$ be an integer and $S=C^{(d)}$. Then we have a natural isomorphism
\[
\operatorname{Hom}_S(T_S,F^{[d]})
\cong H^0(C,\omega_C)\otimes H^0(C,F).
\]
\end{lemma}

\begin{proof}
Set $\Sigma:= \Sigma_{d}(C)$. Let $\widetilde{p}:C\times S\longrightarrow S$ and
$\widetilde{q}:C\times S\longrightarrow C$ be the projections, and
$p=\widetilde{p}|_{\Sigma}$ and $q=\widetilde{q}|_{\Sigma}$. First we show
\begin{equation}\label{[]}
\operatorname{Hom}_{S}\left(T_{S},F^{[d]}\right)
\cong
H^0\left(C,q_*p^*\Omega_{S}\otimes F\right).
\end{equation}
This follows from the following chain of equalities:
\[
\begin{aligned}
H^0(C,q_*p^*\Omega_S\otimes F)
&\cong H^0(C,q_*(p^*\Omega_S\otimes q^*F))\\
&\cong H^0(\Sigma,p^*\Omega_S\otimes q^*F)\\
&\cong H^0(S,p_*(p^*\Omega_S\otimes q^*F))\\
&\cong H^0(S,\Omega_S\otimes F^{[d]})\\
&=H^0(S,\mathcal{H}om_S(T_S,F^{[d]}))\\
&=\operatorname{Hom}_S(T_S,F^{[d]}).
\end{aligned}
\]
Now note that by Lemma \ref{Sx},
\[
h^0(q^{-1}(x),p^*\Omega_S|_{q^{-1}(x)})=h^0(S_x,\Omega_S|_{S_x})
=h^0(S,\Omega_S)
\]
is independent of $x\in C$. So by Grauert's theorem and cohomology and base change,
$q_*p^*\Omega_S$ is a vector bundle on $C$ with
$H^0(S_x,\Omega_S|_{S_x})$ being the fibre over $x\in C$. Also, $\widetilde{q}_*\widetilde{p}^*\Omega_S
\cong H^0(S,\Omega_S)\otimes\mathcal{O}_C$
is a trivial vector bundle on $C$ with each fibre $H^0(S,\Omega_S)$. The natural map
\[
\widetilde{p}^*\Omega_S\longrightarrow
p^*\Omega_S=\widetilde{p}^*\Omega_S\otimes\mathcal{O}_{\Sigma}
\]
induces a map of vector bundles on $C$
\[
\widetilde{q}_*\widetilde{p}^*\Omega_S
\longrightarrow q_*p^*\Omega_S,
\]
which is an isomorphism in each fibre by Lemma \ref{Sx}, hence is an isomorphism. Also, $H^0(S,\Omega_S)=H^0(C,\omega_C),$
the Albanese varieties of $C$ and $S$ being the same. So,
\[
H^0(C,\omega_C)\otimes F\cong H^0(S,\Omega_S)\otimes F
\cong \widetilde{q}_*\widetilde{p}^*\Omega_S\otimes F \cong q_*p^*\Omega_S\otimes F.
\]
Now taking $H^0$ and using \eqref{[]}, we get the result.

\end{proof}
\begin{lemma}\label{homtf1}
Let $F$ be a vector bundle on an elliptic curve $C$ with origin $O$.
Let $d\geq 2$ be an integer and $S'=C^{((d))}.$
Then the following hold:

\begin{enumerate}
    \item[(a)] If $d\geq 3$, then $\operatorname{Hom}_{S'}\left(T_{S'},F^{[[d]]}\right)=0.$

    \item[(b)] If $d=2$, then $\operatorname{Hom}_{S'}\left(T_{S'},F^{[[d]]}\right)
    \cong H^0(C,F(-4\cdot O)).$
\end{enumerate}
\end{lemma}
\begin{proof}
Let $\Sigma'=\Sigma_d'(C),$
and let
\[
p:\Sigma'\longrightarrow S',
\qquad
q:\Sigma'\longrightarrow C
\]
be the projections. The same proof as of \eqref{[]} shows
\begin{equation}\label{[[]]}
\operatorname{Hom}_{S'}\left(T_{S'},F^{[[d]]}\right)
\cong
H^0\left(C,q_*p^*\Omega_{S'}\otimes F\right).
\end{equation}

Note that since $\mathcal O_C(d\cdot O)$ is base-point free, for all $x\in C$ we have $H_x:=p\bigl(q^{-1}(x)\bigr)\subset S'$
is a hyperplane in $S'\cong\mathbb P^{d-1}.$
So, if $d\geq3$, then
\[
h^0\left(q^{-1}(x),
p^*\Omega_{S'}\big|_{q^{-1}(x)}\right)
=
h^0\left(H_x,\Omega_{\mathbb P^{d-1}}\big|_{H_x}\right)
=
h^0\left(H_x,\Omega_{H_x}\oplus\mathcal O_{H_x}(-1)\right)
=0.
\]
Thus, by cohomology and base change, $q_*p^*\Omega_{S'}=0.$
So (a) follows from \eqref{[[]]}.

Now suppose $d=2$. So $q$ is an isomorphism, and
\[
\pi:=p\circ q^{-1}:C\longrightarrow S'\cong\mathbb P^1
\]
is the quotient map given by the complete linear system $|2\cdot O|$. Thus $$q_*p^*\Omega_{S'}
=
\pi^*K_{\mathbb P^1}
\cong \pi^*\mathcal{O}_{\mathbb P^1}(-2)
\cong \mathcal O_C(-4\cdot O).$$
Now (b) follows from \eqref{[[]]}.
\end{proof}

\begin{lemma}\label{multiprojective}
Let $a_1,\ldots,a_r$ be positive integers, and $X=\prod_{i=1}^r \mathbb{P}^{a_i}.$
Then $\operatorname{End}(T_X)\cong \mathbb{C}^r$
as $\mathbb{C}$-algebras.
\end{lemma}

\begin{proof}
Let $p_i:X\longrightarrow \mathbb{P}^{a_i}$
be the projections. So
\[
T_X\cong\bigoplus_{i=1}^r p_i^*T_{\mathbb{P}^{a_i}}.
\]
Since $T_{\mathbb{P}^{a_i}}$ is simple,
\[
\operatorname{End}(p_i^*T_{\mathbb{P}^{a_i}})
\cong
\operatorname{End}(T_{\mathbb{P}^{a_i}})
\cong\mathbb{C}.
\]

Also, for $i\neq j$, we have
\[
\operatorname{Hom}_X
\left(p_i^*T_{\mathbb{P}^{a_i}},
p_j^*T_{\mathbb{P}^{a_j}}\right)
=
H^0\left(
X,
p_i^*\Omega_{\mathbb{P}^{a_i}}
\otimes
p_j^*T_{\mathbb{P}^{a_j}}
\right).
\]
By the K\"unneth formula, this is $H^0(\mathbb{P}^{a_i},\Omega_{\mathbb{P}^{a_i}})
\otimes
H^0(\mathbb{P}^{a_j},T_{\mathbb{P}^{a_j}})$, which is $0$ as $H^0(\mathbb{P}^{a_i},\Omega_{\mathbb{P}^{a_i}})=0.$
These together prove the Lemma.
\end{proof}

\begin{lemma}\label{end=c2}
Let $C$ be an elliptic curve, $d\geq 2$ an integer, and $S=C^{(d)}$. Then $\operatorname{End}(T_S)\cong \mathbb{C}^2$
as $\mathbb{C}$-algebras.
\end{lemma}

\begin{proof}
Let $S\xrightarrow{b}\operatorname{Pic}^d C\cong C$
be the Albanese map, which is a $\mathbb{P}^{d-1}$-bundle. So, $T_S\cong \mathcal{O}_S\oplus T_b.$
 By Proposition \ref{r=1}, $\operatorname{End}(T_b)=\mathbb{C}\cdot\operatorname{id}$. As $b$ is a $\mathbb{P}^{d-1}$-bundle and $H^0(\mathbb{P}^{d-1},\Omega_{\mathbb{P}^{d-1}})=0,$
we get $\operatorname{Hom}_S(T_b,\mathcal{O}_S)=0.$
Since by \cite[Proposition 1.5]{CataneseCiliberto1993} every automorphism of $C^{(d)}$ is natural, we have
\[
\operatorname{Aut}^0(S)\cong \operatorname{Aut}^0(C)\cong C,
\]
hence
\[
h^0(S,T_S)=\dim\operatorname{Aut}^0(S)=1.
\]
Thus
\[
1+h^0(S,T_b)
=h^0(S,\mathcal{O}_S\oplus T_b)
=h^0(S,T_S)
=1,
\]
forcing $h^0(S,T_b)=0.$

These together show $\operatorname{End}(T_S)$ is the product ring $\mathbb{C}^2$.
\end{proof}
\section{Tangent bundle of Quot and Kummer-Quot schemes}
We prove Theorem \ref{A}, Supplement \ref{PE} and Theorem \ref{T_Q'} together in this section. We introduce the following notation, which will be used throughout this section.
Let $r$ be the rank of $E$. If $r=1$, then $Q=C^{(d)}$, in which case Theorem \ref{A} was already proved in Proposition \ref{r=1}, and Theorem \ref{T_Q'} also follows as $Q'$ is a projective space hence $T_{Q'}$ is always simple. So we assume $r\geq 2$ from now on. Let $S=C^{(d)}$ and  $U\subset S$
be the complement of the big diagonal. Let $\varphi:Q\longrightarrow S$
be the Hilbert--Chow morphism, $V=\varphi^{-1}(U),$
and let $f:V\longrightarrow U$
be the restriction of $\varphi$. For $[D]\in U$, the fibre of $\varphi$ over $[D]$ is denoted by $F_D$. Thus, $f$ is smooth, and each $F_D$ is a product of projective spaces.
Let $Q_0\subset Q$
be the set of points where $\varphi$ is smooth. Then $Q_0$ is an open subset containing $V$, and $
\varphi(Q_0)=S$
by \cite[Proposition 3.5]{GangopadhyaySebastianFundamental}. Thus $Q\setminus Q_0$ does not contain any fibre of $\varphi$. As $\varphi$ is flat with irreducible fibres by \cite[Corollary 6.3 and 6.6]{GangopadhyaySebastianFundamental} and $Q\setminus Q_0$ is contained in the preimage of the big diagonal in $C^{(d)}$ by \cite[Lemma 2.3]{Torelli}, we see that $Q\setminus Q_0$ has codimension $\geq 2$ in $Q$. We will use this fact repeatedly. In fact, since fibres of $\varphi$ are normal by \cite[Corollary 6.6]{GangopadhyaySebastianFundamental}, codimension of $Q\setminus Q_0$ is $\geq 3$, and this is an equality if $d\geq 2$ by \cite[Theorem 1.1(4)]{ItoPunctualQuot} and \cite[Corollary 6.5]{GangopadhyaySebastianFundamental}.

Let $\varphi_0:Q_0\longrightarrow S$
be the restriction of $\varphi$. Also, let $K\cong \varphi_*T_{\varphi}$ be the kernel of the natural map $\varphi_*T_Q\to T_S$ as in \cite[Proof of Theorem 4.1]{Torelli}. By \cite[Claim 4.2]{Torelli}, $K\cong (\operatorname{ad} E)^{[d]}.$
\begin{claim}\label{Tf to Tf}
If $\psi\in \operatorname{End}(T_Q)$, then $\psi(T_{\varphi_0})\subset T_{\varphi_0}$
over $Q_0$.
\end{claim}
\begin{proof}
 It suffices to show $\psi(T_f)\subset T_f$
over $V$. Since $F_D$ is a product of projective spaces for $[D]\in U$, we have $\operatorname{Hom}_{F_D}(T_{F_D},\mathcal O_{F_D})=0$. As $\frac{T_Q|_{F_D}}{T_{F_D}}\cong T_S([D])\otimes_{\mathbb{C}}\mathcal{O}_{F_D}$
is trivial, we get
\[
\operatorname{Hom}_{F_D}
\left(
T_{F_D},\frac{T_Q|_{F_D}}{T_{F_D}}\right)=0.
\]
Hence $\psi|_{F_D}(T_{F_D})\subset T_{F_D}$
for all $[D]\in U$, proving the Claim.
\end{proof} 

\begin{claim}\label{no idempotent}
The ring $\operatorname{End}(T_{\varphi_0})$ has no nontrivial idempotent.
\end{claim}
\begin{proof}
As $V$ is a dense open set in $Q_0$, we have an inclusion of algebras $\operatorname{End}(T_{\varphi_0})\hookrightarrow \operatorname{End}(T_f)$. Thus it suffices to show $\operatorname{End}(T_f)$ has no nontrivial idempotent.

Let $\widetilde{U}\subset C^d$ be the preimage of $U$ under the quotient map $C^d\longrightarrow C^{(d)}.$ So, $\pi:\widetilde{U}\longrightarrow U$
is a finite étale Galois cover with Galois group $S_d$. We will show that
\begin{equation}\label{fixed ring}
\operatorname{End}(T_f)
\cong
H^0(\widetilde{U},\mathcal{O}_{\widetilde{U}})^{S_{d-1}},
\end{equation}
where $S_{d-1}\subset S_d$ is the subgroup of permutations fixing $d$. This will complete the proof, as the right-hand side of \eqref{fixed ring} is an integral domain.

Let $\widetilde{V}=V\times_U\widetilde{U},$
and let $\widetilde{V}\xrightarrow{\widetilde f}\widetilde U,$
$\widetilde{V}\xrightarrow{\widetilde\pi}V$ be the projections.
 Since $\widetilde{\pi}$ is finite étale, we have $\widetilde{\pi}^*T_f=T_{\tilde{f}},$ hence $\widetilde{\pi}^*\mathcal{E}nd(T_f)\cong\mathcal{E}nd(T_{\tilde{f}}).$ As $\widetilde\pi$ is the quotient under the free action of $S_d$ on
$\widetilde V$, we get
\begin{equation}\label{end fixed ring}
\operatorname{End}(T_f)
\cong
\operatorname{End}(T_{\widetilde f})^{S_d}.
\end{equation}

If $\widetilde U\subset C^d\xrightarrow{p_i}C$
are the projections, then
\[
\widetilde V
\cong
\mathbb P_{\widetilde U}(p_1^*E)
\times_{\widetilde U}
\mathbb P_{\widetilde U}(p_2^*E)
\times_{\widetilde U}\cdots\times_{\widetilde U}
\mathbb P_{\widetilde U}(p_d^*E)
\] over $\widetilde U.$
This fibre product structure gives a decomposition of $T_{\widetilde f}$ and hence gives an algebra homomorphism
\[
\mathcal O_{\widetilde U}^{\oplus d}
\xlongrightarrow{\eta}
\widetilde f_*\operatorname{\mathcal{E}nd}(T_{\widetilde f}).
\]
Now $\mathcal O_{\widetilde U}^{\oplus d}$ has an $S_d$-action coming
from the $S_d$-action on $\mathcal O_{\widetilde U}$ and the $S_d$-action permuting
the factors. It is easy to see that $\eta$ is $S_d$-equivariant.

Using Lemma \ref{multiprojective}, and cohomology and base change, one sees that
$\eta$ is an isomorphism over each fiber, hence is an isomorphism. So, by \eqref{end fixed ring},
\[
\begin{aligned}
\operatorname{End}(T_f)
&\cong H^0\!\left(\widetilde U,
   \widetilde f_*\operatorname{\mathcal{E}nd}(T_{\widetilde f})\right)^{S_d} \\
&\cong H^0\!\left(\widetilde U,
   \mathcal O_{\widetilde U}^{\oplus d}\right)^{S_d} \\
&\cong
\left(H^0(\widetilde U,\mathcal O_{\widetilde U})^{\oplus d}\right)^{S_d}
\cong H^0(\widetilde U,\mathcal O_{\widetilde U})^{S_{d-1}},
\end{aligned}
\]
proving the Claim. 
\end{proof}
\begin{remark}
   In fact, for $g\geq 2$, we have $H^0(\widetilde U,\mathcal{O}_{\widetilde U})=\mathbb{C}$. Thus, $T_f$ and so $T_{\varphi_0}$ is in fact simple for $g\geq 2$ by \eqref{fixed ring}. Since we will not need this, we just give a sketch of the proof of this fact.

First note that for any finite flat morphism $Y\xrightarrow{\pi} X$ of varieties, the ring extension $H^0(X,\mathcal{O}_X)\subset H^0(Y,\mathcal{O}_Y)$ is integral. Indeed, given $g\in H^0(Y,\mathcal{O}_Y)$, multiplication by $g$ gives an $\mathcal{O}_X$-linear endomorphism of the vector bundle $\pi_*\mathcal{O}_Y$, so its characteristic polynomial is a monic polynomial in $H^0(X,\mathcal{O}_X)[t]$. By the Cayley--Hamilton theorem, $g$ is a root of this polynomial.

Applying this to $\widetilde U\xrightarrow{\pi}U$, it suffices to show $H^0(U,\mathcal{O}_U)=\mathbb{C}$.
If $g\in H^0(U,\mathcal{O}_U)$ is nonconstant, then it has a pole along the big diagonal $\Delta$ of some order $m\geq 1$. So $m\Delta$ is linearly equivalent to the divisor $D$ of zeroes of $g$, which does not contain $\Delta$ in its support. Now if $\eta:=\eta_{d-1,d}\in N_1(C^{(d)})$ is as in \cite[Theorem~4.1]{bastianelli2019effective}, then
$\eta\cdot m\Delta=0$ but $\eta\cdot D>0$, a contradiction.

Consequently, this answers a special case of a question of Leon Takhtajan
(see \cite[Remark 1.10]{chen2026diagonal}).
The general case of that question, asking whether all holomorphic functions
on $\widetilde U$ are constants, seems to be open.
\end{remark}

We now prove Theorem \ref{A}, Supplement \ref{PE} and Theorem \ref{T_Q'} in nine steps.  

\underline{\textbf{Step 1:}} We define a $\mathbb{C}$-algebra homomorphism $\operatorname{End}(T_Q)\xrightarrow{r}\operatorname{End}(T_S)$ with certain properties.

We have a map $
\epsilon:\varphi^*K\longrightarrow T_Q$
obtained by adjunction from the inclusion $K\subset \varphi_*T_Q$. By definition, $d\varphi\circ \epsilon=0,$
where $d\varphi:T_Q\to\varphi^*T_S$ is the natural map. Given
$u\in\operatorname{Hom}_S(T_S,K)$, we define
$N_u\in\operatorname{End}(T_Q)$ to be the composition
\[
T_Q\xrightarrow{d\varphi}\varphi^*T_S
\xrightarrow{\varphi^*u}\varphi^*K
\xrightarrow{\epsilon}T_Q.
\]
This gives a $\mathbb{C}$-linear map
\[
\operatorname{Hom}_S(T_S,K)
\xrightarrow{N}\operatorname{End}(T_Q).
\]

Given $[D]\in U$, $u([D])$ is a map $T_S([D])\longrightarrow H^0(F_D,T_{F_D}),$
and for $w\in F_D$, $N_u(w)$ is the composition
\[
T_Q(w)\xrightarrow{(d\varphi)_w}T_S([D])
\xrightarrow{u([D])}H^0(F_D,T_{F_D})
\longrightarrow T_{F_D}(w)\subset T_Q(w).
\]
As $(d\varphi)_w$ is surjective, from this pointwise description one sees that if $N_u=0$, then
$u([D])=0$ for all $[D]\in U$, hence $u=0$. In other words, $N$ is
injective.

Also, since $d\varphi\circ \epsilon=0$, we get
\begin{equation}\label{square zero}
N_uN_v=0\qquad\text{for all }u,v\in\operatorname{Hom}_S(T_S,K).
\end{equation}

Given $\psi\in\operatorname{End}(T_Q)$, by Claim \ref{Tf to Tf} it induces an
endomorphism of $T_{Q_0}/T_{\phi_0}=\phi_0^*T_S.$ As complement of $Q_0$ has codimension $\geq 2$ in $Q$, and $\phi$ is a contraction, we have $$End(\phi_0^*T_S)=End(\phi^*T_S)=End(T_S).$$
This gives $r(\psi)\in End(T_S)$.
Thus we get a $\mathbb{C}$-algebra homomorphism
\[
\operatorname{End}(T_Q)\xrightarrow{r}\operatorname{End}(T_S).
\]
Clearly $\ker r\supseteq\operatorname{im}N$.

\begin{claim}\label{ker im}
 Suppose $T_Q$ is indecomposable. Then we have
\begin{enumerate}
    \item $\ker r=\operatorname{im}N,$
    \item $\operatorname{im} r=\mathbb{C}\cdot\operatorname{id}.$
  \end{enumerate}  
\end{claim}

\begin{proof}
\underline{$(1)$:} We only need to show $\ker r\subseteq\operatorname{im}N$. Let $\psi\in\ker r$. So $\psi$ is not invertible. As $T_Q$ is indecomposable,  \cite[Lemma 6]{atiyah1956krull} shows $\psi$ is nilpotent. For $[D]\in U$, $\psi$ induces an endomorphism $\psi_D$
of $T_{F_D}$ by Claim \ref{Tf to Tf}, which also must be nilpotent. Now Lemma \ref{multiprojective} forces $\psi_D=0$. So, $\psi(T_f)=0$ on $V$. Hence $\psi(T_{\varphi_0})=0$ on $Q_0$. As $T_{Q_0}/T_{\varphi_0}=\varphi_0^*T_S,$
so $\psi$ induces a homomorphism $\psi_0'\colon \varphi_0^*T_S\longrightarrow T_{Q_0}.$
As the complement of $Q_0$ has codimension $\geq 2$ in $Q$, $\psi_0'$ extends to $\psi'\colon \varphi^*T_S\longrightarrow T_Q.$

The map $d\varphi\colon T_Q\longrightarrow\varphi^*T_S$
induces a map
\[
\operatorname{Hom}_Q(\varphi^*T_S,T_Q)
\xrightarrow{\widetilde{d\varphi}\ }
\operatorname{Hom}_Q(\varphi^*T_S,\varphi^*T_S).
\]
As $\psi\in\ker r$, we have $\widetilde{d\varphi}(\psi')=0$ over $V$, hence $\widetilde{d\varphi}(\psi')=0$. Since $\varphi$ is a contraction, $\widetilde{d\varphi}$ can be identified with the map
\[
\operatorname{Hom}_S(T_S,\varphi_*T_Q)
\longrightarrow
\operatorname{Hom}_S(T_S,T_S),
\]
whose kernel is $\operatorname{Hom}(T_S,K)$. Thus $\psi'$ induces $u\in\operatorname{Hom}(T_S,K).$
Now $N_u$ agrees with $\psi$ on $V$, hence $N_u=\psi$.

\underline{$(2):$} Clearly $\operatorname{im} r\supseteq \mathbb{C}\cdot\operatorname{id}$. Suppose $\operatorname{im} r\neq \mathbb{C}\cdot\operatorname{id}.$
Then $\operatorname{End}(T_S)\neq \mathbb{C}\cdot\operatorname{id}$, hence $d\geq 2$ and by Proposition \ref{r=1} we get $g=1$. Now Lemma \ref{end=c2} forces $r$ to be surjective. As $T_Q$ is indecomposable, by \cite[Lemma 6]{atiyah1956krull} every element of $\operatorname{End}(T_Q)$ is either a unit or nilpotent. The same must be true for $\operatorname{im} r=\operatorname{End}(T_S).$
But this contradicts the description of $\operatorname{End}(T_S)$ as in Lemma \ref{end=c2}.
\end{proof}
\underline{\textbf{Step 2:}}  We use the homomorphism $r$ obtained in Step 1 to show that when $T_Q$ is not indecomposable then so is $T_S$, and there are decompositions of $T_Q$ and $T_S$ which are compatible in some sense. More precisely, we prove the following Claim.
\begin{claim}\label{A+B}
   Suppose $T_Q$ is not indecomposable. Then there are nonzero
subbundles $A,B$ of $T_Q$ and subbundles $\overline{A},\overline{B}$ of $T_S$ satisfying the following:
\begin{enumerate}
\item[(i)] $A\oplus B=T_Q$ and  $\overline{A}\oplus\overline{B}=T_S,$
\item[(ii)]$A|_{Q_0}\supset T_{\varphi_0},$
\item[(iii)] $d\varphi(B)=\varphi^*\overline{B}$, and $d\varphi|_B:B\longrightarrow\varphi^*\overline{B}$ is an isomorphism.
\end{enumerate} Here $d\varphi$ is the tangent map $T_Q\to \varphi^* T_S.$ Moreover $\overline{B}$ is always nonzero, and if $d\geq 2$ then $\overline{A}$ is also nonzero.
\end{claim}
\begin{proof} Write $T_Q=\bigoplus_{i=1}^N A_i,$
where the $A_i$'s are indecomposable subbundles and $N\geq 2$. By Claim \ref{Tf to Tf} and Claim \ref{no idempotent}, one can
easily show that $T_{\phi_0}\subseteq A_i|_{Q_0}$ for some $i$, say $i=1$. Let $A=A_1$ and $B=\bigoplus_{i=2}^N A_i$. So $A$ and $B$ are nonzero and $(ii)$ holds.

The projections onto $A$ and $B$ corresponding to the decomposition
$T_Q=A\oplus B$ give idempotents in $\operatorname{End}(T_Q)$ with sum $1$, whose images
under $r$ give idempotents in $\operatorname{End}(T_S)$ with sum $1$. These give a direct
sum decomposition $T_S=\overline{A}\oplus\overline{B},$
with
\begin{equation}\label{f*}
\varphi_0^*\overline{A}=d\varphi_0(A|_{Q_0})
\quad\text{and}\quad
\varphi_0^*\overline{B}=d\varphi_0(B|_{Q_0}).
\end{equation}

Note that $d\varphi(B)\subseteq\varphi^*\overline{B}$
on the dense open $Q_0$ by \eqref{f*}, so $d\varphi(B)\subseteq\varphi^*\overline{B}.$
By $(ii)$, for each $w\in Q_0$, we have $B(w)\cap T_{\varphi_0}(w)=0$
in $T_Q(w)$. Thus, the map $d\varphi|_B:B\longrightarrow\varphi^*\overline{B}$
is fibrewise injective on $Q_0$. 
Now by \eqref{f*} we get $d\varphi|_B:B\longrightarrow\varphi^*\overline{B}$
is an isomorphism over $Q_0$, hence is an isomorphism on the whole of $Q$,
as $Q\setminus Q_0$ has codimension $\geq 2$ in $Q$. This proves $(iii)$.

It remains to show the last statement. Note that $\overline{B}\neq 0$ by $(iii)$. If $d\geq 2$ and $\overline{A}=0$, then $\overline{B}=T_S$. So by $(iii)$, $d\varphi:T_Q\longrightarrow \varphi^*T_S$
is surjective, hence $\varphi$ is smooth. This contradicts \cite[Lemma 2.3]{Torelli}. Thus $\overline{A}\neq 0$. This completes the proof of the Claim.
\end{proof}
\noindent\underline{\textbf{Step 3:}}
We show that $T_Q$ is indecomposable in cases $(a)$, $(b)$, $(d)$ and $(e)$ of the Theorem \ref{A}.
Suppose $T_Q$ is not indecomposable.
We will show $g=1$ and $E$ is semihomogeneous.
Let $A$, $B$, $\overline{A}$ and $\overline{B}$ be as in Claim \ref{A+B}. By Claim \ref{A+B}, $T_S$ is not indecomposable. So by Proposition \ref{r=1}, we have $g=1$.
Let $S\xrightarrow{\,b\,}C$
be the sum map. By Lemma \ref{end=c2}, there is a unique nontrivial direct sum
decomposition of $T_S$, given by
\[
T_S=T_b\oplus H,
\]
where $H$ is a line subbundle of $T_S$ and
\[
H\xrightarrow{\,db\,}b^*T_C\cong\mathcal O_S
\]
is an isomorphism. So the pair of subbundles $(\overline A,\overline B)$ is
either $(T_b,H)$ or $(H,T_b)$.

If $(\overline A,\overline B)=(H,T_b)$, then $\operatorname{im}(d\varphi)\supseteq \varphi^*T_b$
by Claim \ref{A+B}$(iii)$. Since $T_Q\xrightarrow{\,da\,}a^*T_C$
is surjective by \cite[Lemma 5.1]{Torelli}, we see that the composition
\[
T_Q\xrightarrow{\,d\varphi\,}\varphi^*T_S
\longrightarrow
\frac{\varphi^*T_S}{\varphi^*T_b}=a^*T_C
\]
is surjective. This shows $\operatorname{im}(d\varphi)=\varphi^*T_S$, thus $d\varphi$ is surjective. So $\varphi$ is smooth,
contradicting \cite[Lemma 2.3]{Torelli}.

So $(\overline A,\overline B)=(T_b,H).$
By Claim \ref{A+B}$(iii)$, we have an isomorphism
\[
B\xrightarrow{\,d\varphi|_B\,}\varphi^*\overline B
=\varphi^*H
\xrightarrow[\cong]{\,\varphi^*db\,}
a^*T_C\cong\mathcal O_Q.
\]
Thus, the natural map
\[
H^0(Q,T_Q)\longrightarrow H^0(C,T_C)\cong\mathbb C,
\]
arising as the map on Lie algebras of the group homomorphism
\[
\operatorname{Aut}^0(Q)\xrightarrow{\,\eta\,}\operatorname{Aut}^0(C)
\]
induced by the contraction $a$, is nonzero. So, $\eta$ is nontrivial, hence surjective as $\operatorname{Aut}^0(C)\cong C$.

We have a commutative diagram as in \cite[Proof of Lemma 5.1]{Torelli}:
\begin{equation}\label{times d}
\begin{tikzcd}
\operatorname{Aut}^0(\mathbb P_C(E))
    \arrow[r,"\cong"]
    \arrow[d]
&
\operatorname{Aut}^0(Q)
    \arrow[d, "\eta"]
\\
\operatorname{Aut}^0(C)
    \arrow[r,"\times d"]
&
\operatorname{Aut}^0(C),
\end{tikzcd}
\end{equation}
whose top horizontal arrow is an isomorphism by \cite[Corollary B]{Torelli}.
Thus, the natural map $\operatorname{Aut}^0(\mathbb{P}_C(E))\longrightarrow \operatorname{Aut}^0(C)$
is surjective. In other words, $E$ is semihomogeneous.
\medskip

\noindent\underline{\textbf{Step 4:}} We prove (a) of Supplement \ref{PE}. Note that all indecomposable components of $E$ have the same slope is equivalent to saying that every nonzero direct summand of $E$ has the same slope of $E$, by an easy application of \cite[Theorem 3]{atiyah1956krull}. So by \cite[Theorem~2.4]{biswas2010moduli} and \cite[Section~3]{Infdefor1}, it suffices to show that $T_Q$ is decomposable if and only if the following short exact sequence splits:
\begin{equation}\label{split 1}
0\longrightarrow \varphi_*T_{\varphi}\longrightarrow \varphi_*T_Q\longrightarrow T_C\longrightarrow 0.
\end{equation}

Note that $Hom_C(T_C, \varphi_*T_Q)= Hom_Q(\varphi^*T_C, T_Q)$ by adjunction. So, splitting of \eqref{split 1} is equivalent to splitting of the short exact sequence
\begin{equation}\label{split 2}
0\longrightarrow T_{\varphi}\longrightarrow T_Q\longrightarrow \varphi^*T_C\longrightarrow 0.
\end{equation}

If \eqref{split 2} splits, clearly $T_Q$ is decomposable. Conversely, if $T_Q$ is decomposable, let $A,B,\overline A$ and $\overline B$ be as in Claim \ref{A+B}. We must have $\overline B=T_C$ and $\overline A=0$. Now Claim \ref{A+B}(iii) shows that \eqref{split 2} splits. This proves the equivalence, and also the decomposition of $T_Q$ as in the last statement. Finally, note that $\varphi^*T_C$ is a line bundle, hence simple, and $T_{\varphi}$ is simple by Lemma \ref{simple}.

\medskip

\noindent\underline{\textbf{Step 5:}} We prove that  $T_{Q'}$ is indecomposable in the notations of Theorem \ref{T_Q'}. This is very similar to Step 3. So let $C$ be an elliptic curve, and $S'=C^{((d))}$. Let $\varphi':Q'\longrightarrow S'$
be the Hilbert--Chow morphism, and $K'$ the kernel of the natural map $\varphi'_*T_{Q'}\longrightarrow T_{S'}.$ Let $Q_0'=Q_0\cap Q'$, and $\varphi_0':Q_0'\to S'$ the restriction of $\varphi'$. Also let $U'=U\cap S'$. By the same proof as for $Q\setminus Q_0$ has codimension $\geq 2$ in $Q$, we get $Q'\setminus Q_0'$ has codimension $\geq 2$ in $Q'$.

Similarly as in Step~1, we can define a $\mathbb C$-linear map
\[
\operatorname{Hom}_{S'}(T_{S'},K')
\xrightarrow{\,N'\,}
\operatorname{End}(T_{Q'})
\]
given by
\[
u\longmapsto N'_u,
\]
and satisfying
\[
N'_uN'_v=0
\qquad
\text{for all }
u,v\in\operatorname{Hom}_{S'}(T_{S'},K').
\]
In the same way we defined $r$, we can define a $\mathbb C$-algebra
homomorphism
\[
\bar r:\operatorname{End}(T_{Q'})\longrightarrow\operatorname{End}(T_{S'}).
\]

The same argument as in Claim \ref{Tf to Tf} shows any endomorphism of $T_{Q'}$ carries $T_{\varphi_0'}$ into itself. Let $\widetilde{U'}$ be the preimage of $U'$ under the quotient map $C^d \to C^{(d)}$. Note that $\widetilde{U'}$ is a dense open subset of the set of all $\underline{x}\in C^d$ whose coordinates sum to $0$ in $C$. This latter set is isomorphic to $C^{d-1}$ as any choice of first $d-1$ coordinates uniquely determine the last coordinate. This shows $\widetilde{U'}$ is integral. Now the same proof as in Claim \ref{no idempotent} shows that the ring $End(T_{\varphi_0'})$ has no nontrivial idempotent. Finally, the statement and proof of Claim \ref{A+B} is valid with $(Q, S, Q_0, \varphi,\varphi_0)$ replaced by $(Q', S', Q_0', \varphi',\varphi_0')$. As $S'\cong \mathbb{P}^{d-1}$, we have $T_{S'}$ is indecomposable, by Lemma \ref{simple}. Thus $T_{Q'}$ is indecomposable too.

\medskip

\noindent\underline{\textbf{Step 6:}} We complete proofs of all parts of Theorem \ref{A} except $(c)$. By Step 3, we can assume $T_Q$ is indecomposable. By Claim \ref{ker im}, we have
\[
\operatorname{End}(T_Q)=\mathbb{C}\cdot\operatorname{id}\oplus\operatorname{im}N,
\]
and $\operatorname{im}N$ is a square-zero ideal by \eqref{square zero}.

As $N$ is injective, we have $\operatorname{im}N\cong \operatorname{Hom}_S(T_S,K)$. By Lemma \ref{homtf}, and the isomorphism $K\cong (\operatorname{ad}E)^{[d]}$ as in \cite[Claim 4.2]{Torelli}
we have
\[
\operatorname{im}N
\cong H^0(C,\omega_C)\otimes H^0(C,\operatorname{ad}E).
\]
Note that $E$ is simple if and only if $H^0(C,\operatorname{ad}E)=0$ by \eqref{ad E}. Also note that when on an elliptic curve, any simple or more generally indecomposable vector bundle is semihomogeneous by \cite{atiyah1957vector}
Thus we obtain all parts of Theorem \ref{A} except $(c)$. The last statement of Theorem \ref{A} also follows.
\medskip

\noindent\underline{\textbf{Step 7:}} We prove (b) and (c) of Supplement \ref{PE}. If $T_Q$ is indecomposable, Claim \ref{ker im} shows
\[
\operatorname{End}(T_Q)=\mathbb{C}\cdot\operatorname{id}\oplus\operatorname{im}N,
\]
and $\operatorname{im}N$ is a square-zero ideal by \eqref{square zero}. As $N$ is injective, we have
\[
\operatorname{im}N
\cong \operatorname{Hom}_C(T_C,K)
\cong \operatorname{Hom}_C(T_C,\operatorname{ad}E)
\cong H^0(C,\omega_C\otimes\operatorname{ad}E),
\]
proving (b).

For the ``if'' part of (c), after a line bundle twist we can assume $E=\mathcal{O}_{\mathbb{P}^1}^{\oplus a}\oplus\mathcal{O}_{\mathbb{P}^1}(1)^{\oplus b}$ with $a,b>0$. So, $T_Q$ is indecomposable by $(a)$ of Supplement \ref{PE}. Also, $\operatorname{\mathcal End}(E)\cong E^*\otimes E$ is a direct sum of copies of $\mathcal{O}_{\mathbb{P}^1}(-1)$, $\mathcal{O}_{\mathbb{P}^1}$ and $\mathcal{O}_{\mathbb{P}^1}(1)$. Since $\operatorname{ad}(E)$ is a direct summand of $\operatorname{\mathcal End}(E)$, it is also such a direct sum, by \cite[Theorem 3]{atiyah1956krull}. As $\omega_{\mathbb{P}^1}\cong\mathcal{O}_{\mathbb{P}^1}(-2)$, we get $T_Q$ is simple by (b).

Now we show the ``only if'' part of (c). Suppose $T_Q$ is simple. So, $T_Q$ is indecomposable, hence (a) implies $E$ has at least two indecomposable components. In particular, $E$ is decomposable. So, $h^0(C,\operatorname{ad}E)\neq 0$. If $C\not\cong\mathbb{P}^1$, then $h^0(C,\omega_C)\neq0$, hence $h^0(C,\omega_C\otimes\operatorname{ad}E)\neq0$. Thus $C\cong\mathbb{P}^1$. Since not all indecomposable components of $E$ have the same slopes by (a), $E$ is not trivial up to line bundle twist.

Suppose $E$ is not balanced. So, $E_1:=\mathcal{O}_{\mathbb{P}^1}(a)\oplus\mathcal{O}_{\mathbb{P}^1}(a+b)$ is a direct summand of $E$ for some $b\geq 2$ and $a\in\mathbb{Z}$. So, $\operatorname{\mathcal End}(E_1)$ is a direct summand of $\operatorname{\mathcal End}(E)$. Note that $\mathcal{O}_{\mathbb{P}^1}(b)$ is a direct summand of $\operatorname{\mathcal End}(E_1)\cong E_1^*\otimes E_1$. Thus $\mathcal{O}_{\mathbb{P}^1}(b-2)$ is a direct summand of $\omega_C\otimes\operatorname{\mathcal End}(E)$, and as $b\geq2$, we have $h^0(C,\omega_C\otimes\operatorname{\mathcal End}(E))\neq0$. As $h^0(C,\omega_C)=0$ for $C\cong\mathbb{P}^1$, \eqref{ad E} shows $h^0(C,\omega_C\otimes\operatorname{ad}E)\neq0$.
Thus $T_Q$ is not simple by (b), a contradiction. So, $E$ is balanced.
\medskip

\noindent\underline{\textbf{Step 8:}}  We complete the proof of Theorem \ref{T_Q'}. We use notations of Step 5.

As $S'\cong \mathbb P^{d-1},$
we have $\operatorname{End}(T_{S'})=\mathbb C\cdot\operatorname{id}$
by Lemma \ref{simple}. So,
\[
\operatorname{End}(T_{Q'})
=
\mathbb C\cdot\operatorname{id}\oplus\ker\bar r.
\]

Now $T_{Q'}$ is indecomposable by Step 5. The same proof as of Claim \ref{ker im} (1), with $(T_Q,T_S, r)$ replaced by
$(T_{Q'},T_{S'}, \bar{r})$ everywhere, shows that
\[
\ker\bar r
\cong
\operatorname{Hom}_{S'}(T_{S'},K')
\]
and is a square-zero ideal.

As $K'\cong(\operatorname{ad}E)^{[[d]]}$
by \cite[Claim~4.2$'$]{Torelli}, Lemma \ref{homtf1} gives $\ker \bar r=0$
if $d\geq 3$. This proves (a).

Now suppose $d=2$. By Lemma \ref{homtf1},
\[
\ker\bar r
\cong
H^0\bigl(C,(\operatorname{ad}E)(-4\cdot O)\bigr).
\]
By \eqref{ad E}, this is the same as $H^0\bigl(C,\operatorname{\mathcal End}(E)(-4\cdot O)\bigr).$
Thus,
\[
\ker\bar r
\cong
H^0\bigl(C,E^*\otimes E\otimes\mathcal O_C(-4\cdot O)\bigr)
\cong
H^0\bigl(C,E(4\cdot O)^*\otimes E\bigr)
\cong
\operatorname{Hom}_C(E(4\cdot O),E).
\]
Since $T_{Q'}$ is indecomposable by Step~5, the proof of (b) is complete.

\noindent\underline{\textbf{Step 9:}}  We complete the proof of (c) of Theorem \ref{A}.

By \cite[Lemma 5.1]{Torelli}, we only need to show $T_a$ is simple.
By Lemma \ref{simple}, it suffices to show $T_{a^{-1}(x)}$
is simple for all $x\in C$.

Let $x\in C$. There is $y\in C$ with $dy=x.$
As $E$ is semihomogeneous, there is $\varphi\in\operatorname{Aut}^0(\mathbb P_C(E))$
mapping to $t_y$ under the natural map $\operatorname{Aut}^0(\mathbb P_C(E))
\longrightarrow
\operatorname{Aut}^0(C).$
By \eqref{times d}, we get $\psi\in\operatorname{Aut}^0(Q)$
mapping to $t_x$ under the natural map $\operatorname{Aut}^0(Q)\longrightarrow \operatorname{Aut}^0(C)$. Thus $a^{-1}(x)\cong a^{-1}(O)=Q',$
where $Q'=Q_d'(E,C).$
Thus it suffices to show $T_{Q'}$ is simple.

If $d\geq 3$, then we are done by Theorem \ref{T_Q'}(a).
Now suppose $d=2$. As $E$ is semihomogeneous, by
\cite[Proposition 6.13]{mukai1978semi},
$E$ is semistable. Thus $E(4\cdot O)$ is also semistable, and its slope is strictly
larger than the slope of $E$. So by
\cite[Proposition 1.2.7]{huybrechts2010geometry}, $\operatorname{Hom}_C(E(4\cdot O),E)=0.$
Now we are done by Theorem \ref{T_Q'}(b).
\qed
\begin{corollary}\label{indecomposable}
Let $E$ be a vector bundle of rank $\geq 2$ on a smooth projective curve $C$, and $d\geq 2$ an integer. Then $Q_d(E,C)$ is indecomposable.
\end{corollary}
\begin{proof}
Suppose $Q:=Q_d(E,C)$ is decomposable. We want to get a contradiction. As $T_Q$ is decomposable, we get $C$ is an elliptic curve and $E$ is semihomogeneous, by Theorem \ref{A}.

Write $Q\cong X_1\times X_2$ with $\dim X_1,\dim X_2>0$, and let $p_i:Q\to X_i$ be the projections. So, $T_Q\cong p_1^*T_{X_1}\oplus p_2^*T_{X_2}$.
The decomposition of $T_Q$ as in Theorem \ref{A} (c) and the uniqueness part in \cite[Theorem 3]{atiyah1956krull} force $p_i^*T_{X_i}\cong\mathcal O_Q$ for some $i$, say $i=1$. Restricting to a section of $p_1$, we get $T_{X_1}\cong\mathcal O_{X_1}$, so $X_1$ is an elliptic curve. As formation of Albanese morphism commutes with taking products by \cite[Corollary 4.1.7(2)]{brion2015some}, \cite[Lemma 5.1]{Torelli} forces the Albanese variety of $X_2$ to be a point and hence $p_1$ is the Albanese map of $Q$. So, we can assume $C=X_1$.

As $\operatorname{Aut}^0(C)\times \operatorname{Aut}^0(X_2)
\cong \operatorname{Aut}^0(C\times X_2)$ by \cite[Corollary 2.3]{brion2010automorphism}, the natural map $\operatorname{Aut}^0(Q)\to \operatorname{Aut}^0(C)$ induced by $p_1$ has a section. Now \eqref{times d} shows the map $C\xrightarrow{\times d}C$ has a section, a contradiction as $d\geq2$.
\end{proof}

\section{Distinguishing products of Quot schemes}
We prove Theorem \ref{B} in this section.

\textit{Proof of Theorem \ref{B}:} By \cite[Theorem B and Supplement B]{Torelli}, the conclusion is equivalent to the assertion that $k=l$, and up to renumbering, we have $Q_{d_i}(E_i,C_i)\cong Q_{d_i'}(E_i',C_i')$ for each $i$. Since each factor is indecomposable by Corollary \ref{indecomposable}, in view of  \cite[Theorem~1.2]{horst1985decomposition} we can assume each $Q_{d_i}(E_i,C_i)$ and $Q_{d_j'}(E_j',C_j')$ are torsion bundles over tori.
For a torsion bundle $W$ over a torus with $\dim W>1$, we know that $T_W$ has $\mathcal O_W$ as a direct summand. Thus $T_W$ is decomposable. By Theorem \ref{A}, this implies each $C_i$ and $C_j'$ are elliptic curves, and $E_i$ and $E_j'$ are semihomogeneous.

As observed in Step~9 in the proof of Theorem \ref{A}, each fibre of the Albanese map of $Q_{d_i}(E_i,C_i)$ is isomorphic to $Q_{d_i}'(E_i,C_i)$. Since formation of Albanese morphism commutes with products by \cite[Corollary 4.1.7(2)]{brion2015some}, each fibre of the Albanese map of $\prod_i Q_{d_i}(E_i,C_i)$ is isomorphic to $\prod_i Q_{d_i}'(E_i,C_i)$.

This shows
\[
\prod_{i=1}^k Q_{d_i}'(E_i,C_i)\cong \prod_{j=1}^l Q_{d_j'}'(E_j',C_j').
\]
Now $Q_{d_i}'(E_i,C_i)$ has dimension $>1$ and indecomposable tangent bundle by Theorem \ref{T_Q'}, hence it is not a torsion bundle over a torus. So by \cite[Theorem~1.2]{horst1985decomposition}, we have $k=l$ and, up to renumbering, $Q_{d_i}'(E_i,C_i)\cong Q_{d_i'}'(E_i',C_i')$ for every $i$. Now we are done by \cite[Theorems B$'$ and Supplement B$'$]{Torelli}.
\qed

In fact, when none of $Q_{d_i}(E_i,C_i)$ or $Q_{d_j'}(E_j',C_j')$ is as in (c) of Theorem \ref{A}, any isomorphism between the products comes from isomorphisms of individual factors. Thus we can completely describe the automorphism group of such a product of Quot schemes, as automorphism groups of individual Quot schemes are known by \cite[Corollary B]{Torelli}. This is a consequence of the following general theorems. Strictly speaking, we do not need the implications among $(b),(c)$ and $(d)$ in Theorem \ref{abcd} for our purpose here, we only need the equivalence of $(a)$ and $(b)$ for the proof of Theorem \ref{Urata1}, and the implication $(a)\Rightarrow (d)$ for application to Quot schemes, but that was already observed in \S 2. Nevertheless, we include all the implications for completeness of understanding.

\begin{theorem}\label{abcd}
For a positive dimensional smooth projective variety $X$, consider the following statements:
\begin{enumerate}
\item[(a)] $X$ is a torsion bundle over a torus.
\item[(b)] $\operatorname{Aut}^{0}(X)$ contains a positive dimensional complete subvariety.
\item[(c)] $\operatorname{Aut}^{0}(X)$ is not affine.
\item[(d)]  $T_X$ has $\mathcal{O}_X$ as a direct summand.
\end{enumerate}

Then $(a)\Leftrightarrow(b)\Rightarrow(c)\Leftrightarrow(d)$.
\end{theorem}

\begin{proof}
$(a)\Rightarrow(b)$ was observed in \S 2. $(b)\Rightarrow(c)$ is trivial.  For $(b)\Rightarrow(a)$,  note that if $\operatorname{Aut}^{0}(X)$ contains a positive dimensional complete subvariety $V$, we can assume $V$ contains the identity by translations. The subgroup generated by $V$ is complete, hence is a positive dimensional abelian variety $A$. So, $A$ acts faithfully on $X$. Now as observed in \cite[Introduction]{brion2009some}, there is an $A[n]$-stable subscheme $Y$ of $X$ for some positive integer $n$, such that $X\cong A\times^{A[n]}Y$. This implies $X$ is a torsion bundle over a torus. 

Now we prove $(c)\Leftrightarrow(d)$. Let $G=\operatorname{Aut}^{0}(X)$, and $B$ the quotient of $G$ by its unique maximal connected affine subgroup. So $B$ is an abelian variety, and it is nonzero if and only if $(c)$ holds. Let $A$ be the Albanese variety of $X$. By \cite[Theorem~2.11]{Popa2010DerivedBGG}, we have a homomorphism $\eta:B\to A$ with finite kernel, hence $\eta$ induces an injection of Lie algebras. As the quotient $G\to B$ induces a surjection of Lie algebras, the map $\operatorname{Lie}(G)\xlongrightarrow{\tau} \operatorname{Lie}(A)$
is nonzero if and only if $B$ is nonzero, where $\operatorname{Lie}(H)$ denotes the Lie algebra of an algebraic group $H$. Identifying
$\operatorname{Lie}(G)=H^0(X,T_X)$ and
$\operatorname{Lie}(A)=H^0(X,\Omega_X)^*$, we see that $\tau$ is nonzero if and only if there are
$s\in H^0(X,T_X)$ and $t\in H^0(X,\Omega_X)$ such that $t(s)$ is a nonzero constant. Regarding $s$ as a map $\mathcal{O}_X\to T_X$ and $t$ as a map $T_X\to\mathcal{O}_X$, we see that the last assertion is equivalent to $(d)$. This proves the equivalence.
\end{proof}
\begin{theorem}\label{Urata1}
Let $X_1,\ldots,X_m,Y_1,\ldots,Y_n$ be indecomposable smooth projective varieties, none of them is a torsion bundle over a torus. Then the following holds:

If $\varphi:\prod_i X_i\longrightarrow\prod_j Y_j$ is an isomorphism, then $m=n$, and there is $\sigma\in S_n$ and isomorphisms $X_i\longrightarrow Y_{\sigma(i)}$ inducing $\varphi$.
\end{theorem}
\begin{proof}
If $X'$ is a product of some of the $X_i$'s, then by \cite[Corollary 2.3]{brion2010automorphism} and the equivalence of $(a)$ and $(b)$ in Theorem \ref{abcd}, we see that $\operatorname{Aut}^{0}(X')$ cannot contain a positive-dimensional complete subvariety. Hence the following holds: given a projective variety $Z$ and a morphism $\phi:Z\times X'\to X'$, if $\phi(z_0,\cdot):X'\to X'$ is an isomorphism for some $z_0\in Z$, then $\phi(z,\cdot)=\phi(z_0,\cdot)$ on $X'$ for all $z\in Z$.

 A similar statement holds for products of some of the $Y_j$'s.
Now one can check that the same proof as in \cite[Theorem 1]{urata1981holomorphic} proves the result.
\end{proof}
\begin{remark}
    The implication $(c)\Rightarrow(b)$ in Theorem \ref{abcd} is false, as \cite[Lemma 2.12]{fong2026automorphism} shows.
\end{remark}
\printbibliography
\vspace{30pt}
\begin{flushleft}
{\scshape Department of Mathematics, Shiv Nadar University, NH91, Tehsil
Dadri, Greater Noida, Uttar Pradesh 201314, India}.

{\fontfamily{cmtt}\selectfont
\textit{Email address: ashima.bansal@snu.edu.in} }

\vspace{10pt}

{\scshape Department of Mathematics, Fine Hall, Princeton University, Princeton, NJ 08540, USA}.

{\fontfamily{cmtt}\selectfont
\textit{Email address: ss6663@princeton.edu} }

{\fontfamily{cmtt}\selectfont
\textit{Email address: shivamvatsaaa@gmail.com} }

\end{flushleft}
\end{document}